\documentclass{amsart} 
\usepackage[utf8]{inputenc}
\usepackage{amsmath} 
\usepackage{amsfonts} 
\usepackage{amssymb}
\usepackage{amscd}
\usepackage{amsmath,nccmath}

\usepackage[colorlinks=true]{hyperref}

\usepackage{epsfig}
\usepackage{graphicx}
\usepackage{tikz, tikz-3dplot, pgfplots}
\usepackage[total={6.5in,9.15in},top=1in, bottom=.85in, left=1in, right=1in]{geometry}
\usepackage{fancyhdr}
\usepackage{tikz}
\usetikzlibrary{calc}
\pgfplotsset{compat=1.10}
\usepgfplotslibrary{fillbetween}
\usepackage{xcolor}
\usepackage{enumitem}
\usepackage{comment}
\usepackage{amsthm} 
\theoremstyle{definition}  
\newtheorem{definition}{Definition}[section] 
\newtheorem*{remark}{Remark}
\newtheorem{example}{Example}
\newtheorem{theorem}{Theorem}[section]
\newtheorem{cor}[theorem]{Corollary} 

\makeatletter
\renewcommand*\env@matrix[1][*\c@MaxMatrixCols c]{%
  \hskip -\arraycolsep
  \let\@ifnextchar\new@ifnextchar
  \array{#1}}
\makeatother

\title{A Constructive Proof of NC Fejér-Riesz theorem}
\author{Palak Arora}
\date{July 5, 2022}

\begin{document}


\maketitle

\begin{abstract}
    In this paper we present a constructive proof of Popescu's non-commutative Fejér-Riesz theorem for non-commuting polynomials. We are considering non-commutating polynomial in left-creation and left-annihilation multi-Toeplitz operators.
\end{abstract}

\begin{section}{Introduction}

  The classical Fejér-Riesz states the following: if a trigonometric polynomial
\[
w(e^{it})=\sum_{j=-m}^m c_j e^{ijt}
\]
is nonnegative for all real $t$, then it is expressible in the form
\[
w(e^{it})=|p(e^{it})|^2
\] 
for some analytic polynomial $p(z)=\sum_{j=0}^m a_jz^j$. 
For proof refer to Lemma $2.1$ in \cite{paulsen_2003}. There is also an operator version where the coefficients of $w$ are matrices or operators (\cite{MR227794}). 
The Fej\'{e}r-Riesz theorem can be reformulated as a statement about Toeplitz operators: the function $w$ may be interpreted as the symbol of a {\em Toeplitz operator} $T_w$; in particular if $S$ denotes the unilateral shift on $\ell^2(\mathbb N)$ then $T_w$ is the operator defined by
\[
  T_w = c_0 I +\sum_{k=1}^m c_k S^k + \sum_{k=1}^m c_{-k} S^{*k},
\]
and then the factorization $w=|p|^2$ is equivalent to the factorization of operators
\[
  T_w=T_p^*T_p
\]
where $T_p= \sum_{k=0}^m a_kS^k=p(S)$. (The equivalence of the two formulations follows easily from the fact that $S$ is an isometry ($S^*S=I$)). It turns out this operator formulation admits a generalization, in the noncommutative setting, to so-called {\em multi-Toeplitz} operators, where the single isometry $S$ is replaced by a {\em row isometry}, for example the $d$ tuple of left shifts, $(L_1,\:L_2,\:\cdots,\:L_d)$ or $d$ tuple of right shifts, $(R_1,\:R_2,\:\cdots,\:R_d)$. Precise definitions are given in the next section. Following an idea of Dritschel and Woerdeman \cite{dritschel2005outer}, this paper develops a constructive proof of Riesz-Fejér theorem in the non commutative setting
. We have a non negative multi-Toeplitz polynomial operator
\[
T_Q\::=\:Q_0\otimes I_{\mathcal{F}^2_d}\, +\,\sum_{0<|v|\leq n} Q_v \otimes L_v  \,+\, \sum_{0<|v|\leq n} Q_v^*\otimes L_v^*.
\]
We then find the multi-Toeplitz  operator factorization of this polynomial,
\[
T_Q \: :=\: T_F^*T_F
\]
 where $T_F \: := \sum_{0\leq |v|\leq n} F_v \otimes L_v$ [This is a slight rewording of the theorem $1.6$ \cite{MR1348353}]. 
 \end{section}

\begin{section}{Preliminaries}
Let us recall a few definitions required for the following section:\newline
\begin{definition}
Let $\mathcal{F}_d^+$ denote the word set which is a monoid formed from the letters, $1,\:2,\cdots,\:d$. We say that \textbf{Fock space} is $\ell^2(\mathcal{F}_d^+)$.
\end{definition}


The Fock space, $\ell^2(\mathcal{F}_d^+)$ is the Hilbert space with orthonormal basis $\{\xi_w\}_{w\in \mathcal{F}_d^+}$. 

\begin{definition}
The \textbf{left-creation operator} $L_j$ is defined as
\begin{equation*}
L_j\xi_w=\xi_{jw}
\end{equation*}
for $j=1,\cdots,d$ and can be extended linearly. 
Similarly,  \textbf{left-annihilation operator} $L_j^*$ is defined as

\begin{equation*}
\displaystyle L_j^*\xi_{w}= \begin{cases} \xi_v, & w=jv \\ 0, & \text{otherwise}  \end{cases}
\end{equation*}
for $j=1,\cdots,d$. Thus $\{L_j|j=1,\cdots,d\}$ form a system of isometries with orthogonal ranges: 

\[L_i^*L_j=\delta_{ij}I.\]


\end{definition}\textbf{Note} that same holds for the right shift operators as well.\newline
From the above definition we get that $(L_1,\:L_2,\:\cdots\:,\:L_d)$ and $(R_1,\:R_2,\:\cdots\:,\:R_d)$ are row-isometries.\newline
For any $w\in F_d^+$, $w=i_1i_2\cdots i_n$ we denote $L_w=L_{i_1}L_{i_2}\cdots L_{i_n}$. So for any $w=i_1i_2\cdots i_n$ and $v=j_1j_2\cdots j_m$ in $F_d^+$,\, $L_w^*L_v=L_{i_n}^*\cdots L_{i_2}^*L_{i_1}^*L_{j_1}\cdots L_{j_m}$.
Thus\\
\vskip.15in
$\displaystyle L_w^*L_v= \begin{cases} L_x, &\:\mbox{if } v=wx \\ 
                                       L^*_y, &\: \mbox{if } w=xy\\ 
                                       0, &\: \mbox{otherwise}  
                                       \end{cases}$. \\

 \begin{definition}
In the classical setting, $T$ is said to be a \textbf{Toeplitz operator} if  $S^*TS = T$ where $S$ is unilateral shift. An operator $T$ is \textbf{$L$-multi-Toeplitz} if $L_i^*TL_j=\delta_{ij}T$ where $L_j$ is left-creation operator. Similarly, $T$ is called \textbf{$R$-multi-Toeplitz} if $R_i^*TR_j=\delta_{ij}T$ where $R_j$ is right-creation operator. 
 \end{definition}
       


 \hspace{4mm}
 \begin{example}
 Any left-creation operator $L_w$ is $R$-multi-Toeplitz. Since $L_i$ and $R_j$ commute with each other, $R_j$ commutes with $L_w$ for all $w$ and thus we have 
 \begin{align*}
     R_i^*L_wR_j &= R_i^* R_j L_w\:\\
             &= \delta_{ij} L_w \\
 \end{align*}
 Similarly, $L_v^*$ is $R$-multi-Toeplitz for any word $v$.\\
 Therefore for any non commutative polynomials $f,g$, we have that $f(L)^*+g(L)$ is $R$-multi-Toeplitz.
 \end{example}
 
 Next let us consider a $R$-multi-Toeplitz operator say, 
 \begin{equation}\label{main example}
     T:=\:\,\sum_{0\leq |v|\leq n} q_v  L_v  \,+\, \sum_{0<|v|\leq n} q_v^* L_v^*.
 \end{equation}
 Then corresponding to the Fock space basis $\{\xi_v\}_{v\in \mathcal{F}_d^+}$ we get its matrix representation which is a multi-Toeplitz matrix:\\

 \[ 
 \label{matrixToRefLater}
   \begin{bmatrix}
        \begin{array}{c|ccc|c|ccccccc|cccc}
        
            q_0 & q_1^* & \cdots & q_d^* & \cdots & q_{11\cdots1}^* & \cdots &  q_{1d\cdots d}^* & \cdots &  q_{d1\cdots 1}^* & \cdots & q_{dd\cdots d}^* & 0 & \cdots & 0 & \cdots  \\
            \hline
            q_1 & q_0 & \cdots & 0 & \cdots & q_{1\cdots1}^* & \cdots & 0 & \cdots & q_{d\cdots d}^* & \cdots & 0 & q_{11\cdots1}^* & \cdots & 0 & \cdots\\
            \vdots & \vdots & \ddots & \vdots &  & \vdots & \ddots & \vdots & \ddots & \vdots & \ddots & \vdots & \vdots & \ddots & \vdots & \ddots \\
            q_d & 0 & \cdots & q_0 & \cdots & 0 & \cdots & q_{1\cdots 1}^* & \cdots & 0 & \cdots & q_{d\cdots d}^* & 0 & \cdots & q_{dd\cdots d}^* & \ddots  \\
            \hline

            
            
            
            \vdots & \vdots & \ddots & \vdots & \ddots & \vdots & \ddots & \vdots & \ddots & \ddots & \vdots & \ddots & \vdots & \ddots & \vdots & \ddots  \\

             
              \hline
              q_{11\cdots 1} & q_{1\cdots 1} & \cdots & 0 & \cdots & q_0 & \cdots & 0 & \cdots & 0 & \cdots & 0 & q_1^* & \cdots & 0 & \ddots \\
               \vdots & \vdots & \ddots & \vdots & \vdots & \vdots & \ddots & \vdots &  & \vdots &  & \vdots & \vdots & \ddots & \vdots & \ddots \\
              q_{1d\cdots d} & 0 & \cdots & q_{1\cdots 1} & \cdots & 0 & \cdots & q_0  & \cdots & 0 & \cdots & 0 & 0 &  & 0 &  \\
              \vdots &  & \ddots &  & \ddots & \vdots & & & \ddots & & & \vdots & \vdots & \ddots & \vdots & \ddots \\
             q_{d1\cdots 1} & q_{d\cdots d} & \cdots & 0 & \cdots & 0 & \cdots & 0  & \cdots & q_0 & \cdots & 0 & 0 &  & 0 &  \\
               \vdots & \vdots & \ddots & \vdots & \ddots & \vdots & & & & & \ddots & \vdots & \vdots & \ddots & \vdots & \ddots \\
              q_{dd\cdots d} & 0 & \cdots & q_{d\cdots d} & \cdots & 0 & \cdots & 0  & \cdots & 0 & \cdots & q_0 & 0 & \cdots & q_d^* \\
              \hline
               0 & q_{11\cdots 1} & \cdots & 0 & \cdots & q_1 & \cdots & 0 & \cdots & 0 & \cdots & 0 & q_0 & \cdots & 0 \\
               \vdots &  & \ddots &  & \ddots & \vdots & \ddots & & \ddots & & \ddots & \vdots & \vdots & \ddots & \vdots & \ddots  \\
              0 & 0 & \cdots & q_{dd\cdots d} & \cdots & 0 & \cdots & 0 & \cdots & 0 & \cdots &  q_d & 0 & \cdots & q_0\\
              \hline
              \vdots &  & \ddots &  & \ddots & & \ddots & & \ddots & & \ddots & & & \ddots & & \ddots \\
            \end{array}   
          \end{bmatrix}
        \]
        
\vspace{5mm}
Here we have used the lexicographic ordering for ordering the elements of the word set, $\mathcal{F}_d^+$.

Now we do some relabeling of the indexes here and define for $d=1,\cdots,n$:
\vskip.2in
\[q_k\::=\: col(q_w)_{w\in\mathcal{F}_d^+ ;|w|=k} \,\mbox{    and    }\, q_{-k}\::=\: row(q_w^*)_{w\in\mathcal{F}_d^+ ;|w|=k}\]
\vskip.2in

 and also identifying, $q_1 \otimes I_d:=\begin{pmatrix}
             q_1 & \cdots & 0\\
             \vdots & \ddots & \vdots\\
             0 & \cdots & q_1\\
             \vdots & \ddots & \vdots\\
             \vdots & \ddots & \vdots\\
             q_d & \cdots & 0\\
             \vdots & \ddots & \vdots\\
             0 & \cdots & q_d\\
 \end{pmatrix}$ and so on.

\vskip.2in
Thus we get the following compact form of $T$, which makes it easier to see multi-Toeplitz form of the matrix:
 \vskip.2in
\[T =\begin{bmatrix}
        q_0 & q_{-1} & q_{-2} & \cdots & q_{-n} & 0 & 0 & \cdots\\
        q_{1} & q_0\otimes I_d & q_{-1} \otimes I_d & \cdots & q_{-(n-1)} \otimes I_d  & q_{-n} \otimes I_d & 0 & \cdots\\
         q_{2} & q_{1} \otimes I_d & q_0\otimes I_d\otimes I_d & \cdots &  q_{-(n-2)} \otimes I_d \otimes I_d  & q_{-(n-1)} \otimes I_d \otimes I_d & q_{-n} \otimes I_d^{\otimes 2} & \cdots\\
        \vdots & \vdots & \vdots & \ddots & \vdots  & \vdots & \vdots & \ddots\\
         q_{n} & q_{n-1} \otimes I_d & q_{n-2}\otimes I_d\otimes I_d & \cdots &  q_0 \otimes I_d \otimes \cdots\otimes I_d  & q_{-1} \otimes I_d^{\otimes n} & q_{-2} \otimes I_d^{\otimes n} & \ddots\\
         0 & q_{n} \otimes I_d & q_{n-1}\otimes I_d\otimes I_d & \cdots &  q_1 \otimes I_d^{\otimes n}  & q_{0} \otimes I_d^{\otimes (n+1)} & q_{-1} \otimes I_d^{\otimes (n+1)} & \ddots\\
         0 & 0 & q_{n}\otimes I_d\otimes I_d & \cdots &  q_2 \otimes I_d^{\otimes n}  & q_{1} \otimes I_d^{\otimes (n+1)} & q_{0} \otimes I_d^{\otimes (n+2)} & \ddots\\
         \vdots & \vdots & \vdots & \ddots & \ddots  & \ddots & \ddots & \ddots\\

      \end{bmatrix}.\]
\vskip.2in
\newpage
We are going to use a schur complement technique from Dritschel and Woerdeman \cite{dritschel2005outer} in the proof of the main theorem. So let us define the following:
\begin{definition}
If $\mathcal{H}_1$ and $\mathcal{H}_2$ are Hilbert spaces and 
 \vskip.2in
\[
M=\begin{pmatrix}
            A & B^*\\
            B & C
\end{pmatrix}: \mathcal{H}_1 \oplus \mathcal{H}_2 \rightarrow \mathcal{H}_1 \oplus \mathcal{H}_2
\]
is a positive semidefinite operator then there exists a unique contraction $G:\overline{ran}(C)\rightarrow \overline{ran}(A)$ such that $B=A^{1/2}GC^{1/2}$. The \textbf{Schur complement of $M$ supported on $\mathcal{H}_1$} is defined to be positive semidefinite operator $A^{1/2}(1-GG^*)A^{1/2}$.\newline
An alternative way to define the schur complement of $M$ supported on $\mathcal{H}_1$ is via 
\vskip.2in

\[
\langle Sf,f \rangle = inf \Big\{\Big\langle \begin{pmatrix}
            A & B\\
            B^* & C
\end{pmatrix}\begin{pmatrix}
            f\\
            g
\end{pmatrix},\begin{pmatrix}
            f\\
            g
\end{pmatrix}\Big\rangle:\:g\in \mathcal{H}_2\Big\}
\]
that is, $S:\mathcal{H}_1\longrightarrow \mathcal{H}_1$ is the largest positive semidefinite operator which may be subtracted from $A$ in $M$ such that the resulting operator matrix remains positive semidefinite.\\
\end{definition}
                
\begin{remark}
Consider any positive semidefinite operator matrix, $M$, say\\
\[
M=\begin{pmatrix}
            A & B^*\\
            B & C
\end{pmatrix}: \mathcal{H}_1 \oplus \mathcal{H}_2 \rightarrow \mathcal{H}_1 \oplus \mathcal{H}_2
\]\\
and let $S_M$ be the Schur complement of $M$ supported on $\mathcal{H}_1$. Then for a positive semidefinite matrix, $M\otimes I_d$, the Schur complement supported on $\mathcal{H}_1\otimes \mathbb{C}^d$ is $S_{M\otimes I_d} = S_M \otimes I_d$. \\
\end{remark}
If $Q$ is an operator from $\mathcal{H}$ to $\mathcal{H}$ for some hilbert space $\mathcal{H}$, then $Q\otimes I_d$ takes values from $\mathcal{H}\otimes \mathbb{C}^d$ and outputs in $\mathcal{H}\otimes\mathbb{C}^d$.\vspace{2mm}

Let us denote $\mathcal{H}_i=\mathcal{H}\otimes(\mathbb{C}^d)^{\otimes i}$ for $i\geq 0$.\vspace{2mm}

We make use of the following notation for the next section from \cite{dritschel2005outer}: Typically we will index rows and columns of an $n\times n$ matrix with $0,\:\dots,n-1$. For $\Lambda\subseteq \{0,\cdots,n-1\}$ and an $n\times n$ matrix $M$, we write $S(M;\Lambda)$, or $S(\Lambda)$ when there is no chance of confusion, for the Schur complement supported on the rows and columns labeled by elements of $\Lambda$. It is usual to view $S(\Lambda)$ as an $m\times m$ matrix, where $m=card \Lambda$, however it is often useful to take $S(\Lambda)$ as an $n\times n$ matrix by padding rest of the entries in this $n\times n$ matrix with zeros. For notational convenience we have used $S(m)$ for Schur complement supported on rows and columns labeled by $\{0,\:\cdots,m\}$. 

\end{section}

\begin{section}{MAIN THEOREM}
\vspace{2mm}

The following Theorem $3.1$ and Corollary $3.2$ are multi-Toeplitz versions of Proposition $3.1$ and corollary $3.2$ from \cite{dritschel2005outer}.
\begin{theorem}\label{thm 1}

Consider the positive semidefinite multi-Toeplitz operator matrix\\

\[
T_Q = \begin{bmatrix}
        Q_0 & Q_{-1} & Q_{-2} & \cdots & \cdots \\
        Q_{1} & Q_0\otimes I_d & Q_{-1} \otimes I_d & \cdots & \cdots \\
        Q_{2} & Q_{1} \otimes I_d & Q_0\otimes I_d\otimes I_d & \cdots & \cdots \\
        \vdots & \vdots & \vdots & \ddots &   &  \\
      \end{bmatrix}
\]\\

acting on $\mathcal{H}_0 \oplus \mathcal{H}_1 \oplus \mathcal{H}_2 \oplus \cdots$.
Then Schur complement of $T_Q$ satisfies the reccurrence relation:\\
\[
    S(m)=\begin{bmatrix}
            A & B^* \\
            B & S(m-1)\otimes I_d
        \end{bmatrix}
\]\\
for appropriate choice of $A: \mathcal{H}_0 \rightarrow \mathcal{H}_0$ and $B^*: \mathcal{H}_1 \oplus \mathcal{H}_2 \oplus \cdots \rightarrow \mathcal{H}$. 
When $Q_j=0$ for $j\geq m+1$, then $A=Q_0$ and $B=col(Q_i)_{i=1}^m$.
\end{theorem}

\begin{remark}
Given \[
T_Q = \begin{bmatrix}
        Q_0 & Q_{-1} & Q_{-2} & \cdots & \cdots \\
        Q_{1} & Q_0\otimes I_d & Q_{-1} \otimes I_d & \cdots & \cdots \\
        Q_{2} & Q_{1} \otimes I_d & Q_0\otimes I_d\otimes I_d & \cdots & \cdots \\
        \vdots & \vdots & \vdots & \ddots &   &  \\
      \end{bmatrix},
\]\\

we observe that $T_Q$ can be identified with
\vskip.2in
\[
T_Q = \begin{bmatrix}
       Q_0 & row(Q_{-j})_{j\geq 1} \\
       col(Q_j)_{j\geq 1} & T_Q \: \otimes \: I_d

       \end{bmatrix}.
 \]
 \vskip.2in

\end{remark}

\begin{proof}
Let us write
\[
     S(m) = \begin{bmatrix}
             A & B^* \\
             B & C
            \end{bmatrix}: \mathcal{H}_0 \oplus \mathcal{H}_1 \oplus \mathcal{H}_2 \oplus \cdots \oplus \mathcal{H}_m \rightarrow \mathcal{H}_0 \oplus \mathcal{H}_1 \oplus \mathcal{H}_2 \oplus \cdots\, \oplus\, \mathcal{H}_m.
\]

By definition of Schur complement we have that 
\[
    T_Q - \begin{bmatrix}
            S(m) & 0\\
            0 & 0 
          \end{bmatrix} \geq 0.
\]

That is, we have that
\begin{align}\label{first step}
\begin{bmatrix}
 \begin{array}{c|c}
       Q_0 & row(Q_{-j})_{j\geq 1} \\\hline
       col(Q_j)_{j\geq 1} & T_Q \: \otimes \: I_d
        \end{array}
       \end{bmatrix} - \begin{bmatrix}\begin{array}{c|cc}
            A & B^* & 0\\\hline
            B & C & 0 \\
            0 & 0 & 0
           \end{array} 
          \end{bmatrix} \geq 0 .
 \end{align}
 \vskip.2in

Then leaving out $0$th row and column in (\ref{first step}) we get,\\

\[
T_Q \otimes I_d - \begin{bmatrix}
            C & 0\\
            0 & 0 
          \end{bmatrix} \geq 0.
\]\\

So we have from theorem \ref{thm 1}, $C\leq S(m-1)\otimes I_d$.

Now leaving out the rows and columns $1, \dots, m$ in $(\ref{first step})$, we get\\
\[
\begin{bmatrix}
            Q_0-A & row(Q_j^*)_{j\geq m+1}\\
            \:col(Q_j)_{j\geq m+1} & T_Q \otimes I_d^{m+1} 
          \end{bmatrix} \geq 0.
\]\\

That is, \\

\[
A\leq S\left(\begin{bmatrix}
            Q_0 & row(Q_j^*)_{j\geq m+1}\\
            \:col(Q_j)_{j\geq m+1} & T_Q \otimes I_d^{m+1} 
          \end{bmatrix};0\right) := \Tilde{A}.
\]
\vskip.2in
Note that when $Q_j=0$, $j \geq m+1$ then $\Tilde{A} = Q_0$.\\

Again considering the following operator matrix
\vskip.2in
\begin{align}\label{X}
    \begin{bmatrix}
            Q_0-\Tilde{A} & X & row(Q_j^*)_{j\geq m+1}\\
            X^* &  (Q_{i-j} \otimes I_d)_{i=1,j=1}^m-S(m-1)\otimes I_d & (Q_{i-j}\otimes I_d)_{i=1,j=m+1}^{m+1,\infty}  \\
            col(Q_j)_{j\geq m+1} & (Q_{i-j} \otimes I_d)_{i=m+1,j=1}^{\infty,m} & T_Q\otimes I_d^{m+1}
          \end{bmatrix}
\end{align}\\

The existence of an operator $X$ making this into a positive semidefinite matrix is a variant of a standard operator matrix completion problem, (see Theorem $XVI.3.1$ in \cite{MR1120546})
\: so there always exists such an X. \\
Note that when $\Tilde{A}=Q_0$ we have necessarily that $X=0$. We fix such an $X$. Now $(\ref{X})$ is positive semidefinite, we obtain that \\
\[
\begin{bmatrix}
            \Tilde{A} & row(Q_j^*)_{j=1}^m - X\\
            col(Q_j)_{j=1}^m - X & S(m-1)\otimes I_d
          \end{bmatrix}\leq S(m) = \begin{bmatrix}
                      A & B^*\\
                      B & C
          \end{bmatrix}.
\]
\vskip.2in
This implies that $\Tilde{A}\leq A$ and $S(m-1)\otimes I_d \leq C$. From above we also have $A\leq \Tilde{A}$ and $C\leq S(m-1)\otimes I_d$, thus the equalities  $A=\Tilde{A}$ and $C=S(m-1)\otimes I_d$ follow. \\
Moreover, if $Q_j=0$ for $j\geq m+1$, we have that $\Tilde{A}=Q_0$ and $X=0$, and thus $B=col(Q_i)_{i=1}^m$.
\end{proof}

\begin{cor} \label{cor}
Consider the positive semidefinite multi-Toeplitz operator matrix
\[
T_Q = \begin{bmatrix}
        Q_0 & Q_{-1} & Q_{-2} & \cdots & \cdots \\
        Q_{1} & Q_0\otimes I_d & Q_{-1} \otimes I_d & \cdots & \cdots \\
        Q_{2} & Q_{1} \otimes I_d & Q_0\otimes I_d\otimes I_d & \cdots & \cdots \\
        \vdots & \vdots & \vdots & \ddots &   &  \\
      \end{bmatrix}
\]
\vskip.2in
acting on $\mathcal{H}_0 \oplus \mathcal{H}_1 \oplus \mathcal{H}_2 \oplus \cdots$. Then for each $m \geq 0$, there exist operators $F_0$, $F_1$, $\cdots$ with
\[
  F_{m-k}\otimes I_d^{\otimes k}: \mathcal{H}_k \longrightarrow \overline{ran}F_0 \otimes \mathbb{C}^m \subseteq \mathcal{H}\otimes\mathbb{C}^m
\]

for $0\leq k \leq (m-1)$ so that the Schur complements $S(m)$ of $T_Q$ satisfy

\[
S(m)=\begin{bmatrix}
            F_0^* & F_1^* & F_2^* & \cdots & F_m^*\\
                  & F_0^*\otimes I_d & F_1^*\otimes I_d & \cdots & F_{m-1}^*\otimes I_d\\
                  &  &  F_0^*\otimes I_d^{\otimes 2} & \cdots & F_{m-2}^*\otimes I_d^{\otimes 2}\\
                  &  &  &  \ddots & \vdots\\
                  &  &  &  &  F_0^*\otimes I_d^{\otimes m}
\end{bmatrix}\begin{bmatrix}
            F_0 &  &  &  &\\
            F_1 & F_0\otimes I_d &  &  &\\
            F_2 & F_1\otimes I_d & F_0\otimes I_d^{\otimes 2} &  &\\
            \vdots & \vdots & \vdots & \ddots \\
            F_m & F_{m-1}\otimes I_d & F_{m-2}\otimes I_d^{\otimes 2} & \cdots & F_0\otimes I_d^{\otimes m}
\end{bmatrix}.
\]
\end{cor}
\vskip.2in

\begin{proof}
We will prove this by induction on \textbf{$m$}.\\
\textbf{Base step:} $S(0)$ being a positive semidefinite operator, we can write $S(0)=F_0^*F_0$ where $F_0=(S(0))^{1/2}$.\newline
\textbf{Induction hypothesis:} Let us assume that the result holds for $S(m-1)$.\newline
By \cite{dritschel2005outer}, Proposition 3.1, 
\: we have that $(S(m))_{m,m}=(S(m-1))_{m-1,m-1}\otimes I_d=F_0^*F_0\otimes I_d^{\otimes (m-1)}\otimes I_d=F_0^*F_0\otimes I_d^{\otimes m}$.\newline

From \cite{dritschel2005outer} Corollary 2.3, we have that $S(m-1)=S(S(m);m-1)$. 
 Thus applying Lemma 2.1 from \cite{dritschel2005outer} to
\begin{align}
  P=  \begin{bmatrix}
            F_0 &  &  &  &\\
            F_1 & F_0\otimes I_d &  &  &\\
            F_2 & F_1\otimes I_d & F_0\otimes I_d^{\otimes 2} &  &\\
            \vdots & \vdots & \vdots & \ddots \\
            F_{m-1} & F_{m-2}\otimes I_d & F_{m-3}\otimes I_d^{\otimes 2} & \cdots & F_0\otimes I_d^{\otimes (m-1)}
\end{bmatrix},\, R=F_0\otimes I_d^{\otimes m}
\end{align}\\
there exist $(G_m\: \cdots\: G_1)$ 
so that
\vskip.2in
\[
S(m)=\begin{bmatrix}
            F_0^* & F_1^* & F_2^* & \cdots & G_m^*\\
                  & F_0^*\otimes I_d & F_1^*\otimes I_d & \cdots & G_{m-1}^*\\
                  &  &  F_0^*\otimes I_d^{\otimes 2} & \cdots & G_{m-2}^*\\
                  &  &  &  \ddots & \vdots\\
                  &  &  &  &  F_0^*\otimes I_d^{\otimes m}
\end{bmatrix}\begin{bmatrix}
            F_0 &  &  &  &\\
            F_1 & F_0\otimes I_d &  &  &\\
            F_2 & F_1\otimes I_d & F_0\otimes I_d^{\otimes 2} &  &\\
            \vdots & \vdots & \vdots & \ddots \\
            G_m & G_{m-1} & G_{m-2} & \cdots & F_0\otimes I_d^{\otimes m} 
\end{bmatrix},
\]
and $ran(G_m\: \cdots\: G_1)\subseteq \overline{ran} F_0\otimes (\mathbf{C}^d)^{\otimes m}$. 
Comparing with $S(m) = \begin{bmatrix}
                      A & B^*\\
                      B & S(m-1)\otimes I_d
          \end{bmatrix}$ along with the induction hypothesis yields, $S(m-1)\otimes I_d$ factors into\\

\begin{align*}
  \begin{bmatrix}
                   F_0^*\otimes I_d & F_1^*\otimes I_d & \cdots & F_{m-2}^*\otimes I_d & G_{m-1}^*\\
                    &  F_0^*\otimes I_d^{\otimes 2} & \cdots & F_{m-3}^*\otimes I_d^{\otimes 2} & G_{m-2}^*\\
                    &  &  \ddots & \vdots & \vdots\\
                    &  &  & F_0^*\otimes I_d^{\otimes (m-1)} & G_{1}^*\\
                    &  &  &  & F_0^*\otimes I_d^m
\end{bmatrix}\begin{bmatrix}
            F_0\otimes I_d &  &  &  &\\
            F_1\otimes I_d & F_0\otimes I_d^{\otimes 2} &  &  &\\
            F_2\otimes I_d & F_1\otimes I_d^{\otimes 2} & F_0\otimes I_d^{\otimes 3} &  &\\
            \vdots & \vdots & \vdots & \ddots \\
            G_{m-1} & G_{m-2} & G_{m-3} & \cdots & F_0\otimes I_d^{\otimes m} 
\end{bmatrix}\\
\end{align*}

\vskip.2in

\begin{align*}
&= \begin{bmatrix}
                   F_0^*\otimes I_d & F_1^*\otimes I_d & \cdots & F_{m-2}^*\otimes I_d & F_{m-1}^*\otimes I_d\\
                    &  F_0^*\otimes I_d^{\otimes 2} & \cdots & F_{m-3}^*\otimes I_d^{\otimes 2} & F_{m-2}^*\otimes I_d^{\otimes 2}\\
                    &  &  \ddots & \vdots & \vdots\\
                    &  &  & F_0^*\otimes I_d^{\otimes (m-1)} & F_{1}^*\otimes I_d^{\otimes(m-1)}\\
                    &  &  & & F_0^*\otimes I_d^{\otimes m}
\end{bmatrix}\begin{bmatrix}
            F_0\otimes I_d &  &  &  &\\
            F_1\otimes I_d & F_0\otimes I_d^{\otimes 2} &  &  &\\
            F_2\otimes I_d & F_1\otimes I_d^{\otimes 2} & F_0\otimes I_d^{\otimes 3} &  &\\
            \vdots & \vdots & \vdots & \ddots \\
            F_{m-1}\otimes I_d & F_{m-2}\otimes I_d^{\otimes 2} & F_{m-3}\otimes I_d^{\otimes 3} & \cdots & F_0\otimes I_d^{\otimes m}
\end{bmatrix}\\               
\end{align*}
and thus we have 
\vskip.2in
\[
F_0^*\otimes I_d^{\otimes m}\:\hspace{3mm}(\hspace{3mm} G_{m-1} \hspace{3mm} G_{m-2} \hspace{3mm} \cdots \hspace{3mm} G_1\hspace{3mm})\hspace{3mm} = F_0^*\otimes I_d^{\otimes m} \hspace{3mm} (\hspace{3mm}F_{m-1}\otimes I_d \hspace{3mm} F_{m-2}\otimes I_d^{\otimes 2} 
\hspace{3mm} \cdots \hspace{3mm} F_1\otimes I_d^{\otimes(m-1)}\hspace{3mm}).
\]\vskip.2in As \[{ran}(\hspace{3mm}G_{m-1} \hspace{3mm} G_{m-2}\hspace{3mm} \cdots \hspace{3mm} G_1\hspace{3mm})\hspace{3mm}\subseteq\hspace{3mm} \overline{ran}F_0\otimes (\mathbb{C}^d)^{\otimes m}\] and\\
\vskip.2in
\[ran(\hspace{3mm}F_{m-1}\otimes I_d \hspace{3mm} F_{m-2}\otimes I_d^{\otimes 2} \hspace{3mm} F_{m-3}\otimes I_d^{\otimes 3}\hspace{3mm} \cdots \hspace{3mm} F_1\otimes I_d^{\otimes (m-1)}\hspace{3mm})\hspace{3mm}\subseteq \hspace{3mm}\overline{ran}F_0\otimes (\mathbb{C}^d)^{\otimes m}\] 
\vskip.2in
it follows that $G_j=F_j\otimes I_d^{\otimes (m-j)}$ for $j=1,\cdots,m-1$.
By setting $F_m:= G_m$, we have our result.
\end{proof}
\vskip.2in

\begin{theorem}
If $(L_1,\,L_2,\,\cdots,\,L_d)$ is the left $d$-shift,  
$\mathcal{H}$ a Hilbert space and  $\{Q_w\}_{|w|\leq n}$ are operators $Q_w:\:\mathcal{H}\rightarrow \mathcal{H}$ such that the operator $T_Q: \mathcal{H}\otimes \ell^2(\mathcal{F}^+_d) \rightarrow \mathcal{H}\otimes \ell^2(\mathcal{F}^+_d)$ given by\\
\vskip.2in
\[
T_Q:=\:Q_0\otimes I_{\ell^2(\mathcal{F}^+_d)}\, +\,\sum_{0<|v|\leq n} Q_v \otimes L_v  \,+\, \sum_{0<|v|\leq n} Q_v^*\otimes L_v^*\, 
\]
is positive, then there exist operators $F_0,\cdots,F_w:\mathcal{H}\rightarrow\mathcal{H}$, $|w|\leq n$, such that for $T_F:=\,\sum_{0\leq|v|\leq n} F_v\otimes L_v$ we have $T_Q=T_F^*T_F$.
\end{theorem}

\begin{proof}
Note that this operator, $T_Q$ is $R$ multi-Toeplitz operator. Let us consider the matrix representation of this operator, $T_Q$ corresponding to basis, $\beta=\{e_i\otimes \xi^v\}_{i,v}$ where $\{e_i\}_{i}$ is some orthonormal basis for $\mathcal{H}$ and  $v\in\mathcal{F}_d^+$. It is a multi-toeplitz operator matrix (corresponding to polynomial (\ref{main example})) as was defined on page \pageref{matrixToRefLater} with operator entries $Q_j$ for $j\in\mathcal{F}_d^+$ instead. Then as done before, after relabeling the indexes we get the multi-Toeplitz matrix:

\vskip.2in
 \[T_Q =\begin{bmatrix}
        Q_0 & Q_{-1} & Q_{-2} & \cdots & Q_{-n} & 0 & 0 & \cdots\\
        Q_{1} & Q_0\otimes I_d & Q_{-1} \otimes I_d & \cdots & Q_{-(n-1)} \otimes I_d  & Q_{-n} \otimes I_d & 0 & \cdots\\
         Q_{2} & Q_{1} \otimes I_d & Q_0\otimes I_d\otimes I_d & \cdots &  Q_{-(n-2)} \otimes I_d \otimes I_d  & Q_{-(n-1)} \otimes I_d \otimes I_d & Q_{-n} \otimes I_d^{\otimes 2} & \cdots\\
        \vdots & \vdots & \vdots & \ddots & \vdots  & \vdots & \vdots & \ddots\\
         Q_{n} & Q_{n-1} \otimes I_d & Q_{n-2}\otimes I_d\otimes I_d & \cdots &  Q_0 \otimes I_d \otimes \cdots\otimes I_d  & Q_{-1} \otimes I_d^{\otimes n} & Q_{-2} \otimes I_d^{\otimes n} & \ddots\\
         0 & Q_{n} \otimes I_d & Q_{n-1}\otimes I_d\otimes I_d & \cdots &  Q_1 \otimes I_d^{\otimes n}  & Q_{0} \otimes I_d^{\otimes (n+1)} & Q_{-1} \otimes I_d^{\otimes (n+1)} & \ddots\\
         0 & 0 & Q_{n}\otimes I_d\otimes I_d & \cdots &  Q_2 \otimes I_d^{\otimes n}  & Q_{1} \otimes I_d^{\otimes (n+1)} & Q_{0} \otimes I_d^{\otimes (n+2)} & \ddots\\
         \vdots & \vdots & \vdots & \ddots & \ddots  & \ddots & \ddots & \ddots\\
         
      \end{bmatrix}.\]
\vskip.2in
In this case we have that $Q_{j}=0$ for $|j|\geq n+1$.\newline
\vspace{2mm}\newline
Now consider the schur complement, $S(n)=\begin{bmatrix}
            A & B^* \\
            B & C
\end{bmatrix}$ (that is, supported on first $n+1$ rows and columns of $T_Q$). \\
\vskip.2in

Thus from theorem $\ref{thm 1}$ we get that 
\[
A=Q_0\mbox{ and }B = col(Q_i)_{i=1}^{n}.
\]
\vskip.2in

Comparing the first row of $S(n)$ above with the first row of the product  in corollary $\ref{cor}$ factorization we get that:\\
\vskip.2in
\begin{align*}
   Q_0 &= \sum_{j=0}^n F_j^*F_j;\\
Q_{-j} &= \sum_{k=j}^{n}F_k^* (F_{k-j}\otimes I_d^{\otimes j})\mbox{ for }1\leq j \leq n. 
\end{align*}

\vskip.2in

These $F_i$ for $i=0,\:\cdots\:,n$ are known from corollary $\ref{cor}$.




    

Due to the Toeplitz structure and self adjointness of the matrix, $T_Q$, the above equations give us the following factorization:
\[
 T_Q = T_F^*T_F 
\]
where 
{\fontsize{10}{4}
\[
 T_F\, =\, \begin{pmatrix}
            F_0 &  &  &  & & & & & &\\
            F_1 & F_0\otimes I_d &  &  & & & & & &\\
            F_2 & F_1\otimes I_d & F_0\otimes I_d^{\otimes 2} &  & & & & & &\\
            \vdots & \vdots & \vdots & \ddots & & & & &\\
            F_n & F_{n-1}\otimes I_d & F_{n-2}\otimes I_d^{\otimes 2} & \cdots & F_0\otimes I_d^{\otimes n} & & & & &\\
            0 & F_{n}\otimes I_d & F_{n-1}\otimes I_d^{\otimes 2} & \cdots & F_1\otimes I_d^{\otimes n} & F_0\otimes I_d^{\otimes (n+1)} & & & &\\
            0 & 0 & F_{n}\otimes I_d^{\otimes 2} & \cdots & F_2\otimes I_d^{\otimes n} & F_1\otimes I_d^{\otimes (n+1)} & F_{0}\otimes I_d^{\otimes (n+2)}& & &\\
            \vdots & \vdots & \ddots & \ddots & \ddots & \ddots & \ddots & \ddots & &
            \end{pmatrix}.
 \]}
 \vskip.2in

In compact representation of $T_F$, recall $F_i = col(F_v)_{|v|=i}$. Thus the operator corresponding to the above matrix is,
 \[T_F = \sum_{0\leq |v|\leq n}F_v\otimes L_v.\]
 \vskip.2in
 

 Therefore, we have a multi-toeplitz factorization for a multi-toeplitz positive semidefinite matrix, $T_Q$ which is of the form $T_F^*T_F$.
\end{proof}
\vskip.2in

$\textbf{Remark}$: The above polynomial operator, $T_Q$ (since its a polynomial in the left shifts, $L$)  
is $R$-multi-Toeplitz operator. 
Analogously, we have similar result for the polynomial operators in right shifts.
\end{section}
\vskip.2in
\bibliographystyle{IEEEtran}
 \bibliography{bibliography.bib}

\begin{thebibliography}{1}
\providecommand{\url}[1]{#1}
\csname url@samestyle\endcsname
\providecommand{\newblock}{\relax}
\providecommand{\bibinfo}[2]{#2}
\providecommand{\BIBentrySTDinterwordspacing}{\spaceskip=0pt\relax}
\providecommand{\BIBentryALTinterwordstretchfactor}{4}
\providecommand{\BIBentryALTinterwordspacing}{\spaceskip=\fontdimen2\font plus
\BIBentryALTinterwordstretchfactor\fontdimen3\font minus
  \fontdimen4\font\relax}
\providecommand{\BIBforeignlanguage}[2]{{%
\expandafter\ifx\csname l@#1\endcsname\relax
\typeout{** WARNING: IEEEtran.bst: No hyphenation pattern has been}%
\typeout{** loaded for the language `#1'. Using the pattern for}%
\typeout{** the default language instead.}%
\else
\language=\csname l@#1\endcsname
\fi
#2}}
\providecommand{\BIBdecl}{\relax}
\BIBdecl

\bibitem{paulsen_2003}
V.~Paulsen, \emph{Completely Bounded Maps and Operator Algebras}, ser.
  Cambridge Studies in Advanced Mathematics.\hskip 1em plus 0.5em minus
  0.4em\relax Cambridge University Press, 2003.

\bibitem{MR227794}
\BIBentryALTinterwordspacing
M.~Rosenblum, ``Vectorial {T}oeplitz operators and the {F}ej\'{e}r-{R}iesz
  theorem,'' \emph{J. Math. Anal. Appl.}, vol.~23, pp. 139--147, 1968.
  [Online]. Available: \url{https://doi.org/10.1016/0022-247X(68)90122-4}
\BIBentrySTDinterwordspacing

\bibitem{dritschel2005outer}
M.~Dritschel and H.~Woerdeman, ``Outer factorizations in one and several
  variables,'' \emph{Transactions of the American Mathematical Society}, vol.
  357, no.~11, pp. 4661--4679, 2005.

\bibitem{MR1348353}
\BIBentryALTinterwordspacing
G.~Popescu, ``Multi-analytic operators on {F}ock spaces,'' \emph{Math. Ann.},
  vol. 303, no.~1, pp. 31--46, 1995. [Online]. Available:
  \url{https://doi.org/10.1007/BF01460977}
\BIBentrySTDinterwordspacing

\bibitem{MR1120546}
\BIBentryALTinterwordspacing
C.~Foias and A.~E. Frazho, \emph{The commutant lifting approach to
  interpolation problems}, ser. Operator Theory: Advances and
  Applications.\hskip 1em plus 0.5em minus 0.4em\relax Birkh\"{a}user Verlag,
  Basel, 1990, vol.~44. [Online]. Available:
  \url{https://doi.org/10.1007/978-3-0348-7712-1}
\BIBentrySTDinterwordspacing

\end{thebibliography}

\end{document}